\documentclass[11pt,reqno]{amsart}
\usepackage{amssymb,mathrsfs,graphicx}
\usepackage{ifthen}
\usepackage{mathtools}
\mathtoolsset{showonlyrefs=true}
\usepackage{colortbl}
\definecolor{black}{rgb}{0.0, 0.0, 0.0}
\definecolor{red}{rgb}{1.0, 0.5, 0.5}

\provideboolean{shownotes} 
\setboolean{shownotes}{true} 
\newcommand{\margnote}[1]{
\ifthenelse{\boolean{shownotes}}%
{\marginpar{\raggedright\tiny\texttt{#1}}}%
{}%
}
\newcommand{\hole}[1]{
\ifthenelse{\boolean{shownotes}}%
{\begin{center} \fbox{ \rule {.25cm}{0cm} \rule[-.1cm]{0cm}{.4cm}
\parbox{.85\textwidth}{\begin{center} \texttt{#1}\end{center}} \rule
{.25cm}{0cm}}\end{center}} {} }

\title[1D Eulerian dynamics with singular
  interaction forces]{Global regularity for 1D Eulerian dynamics with singular
  interaction forces}

\author[Alexander Kiselev]{Alexander Kiselev}
\address[Alexander Kiselev]{\newline Department of Mathematics, \ Duke
  University, 120 Science Drive, Durham, NC 27708, USA}
\email{kiselev@math.duke.edu}

\author[Changhui Tan]{Changhui Tan}
\address[Changhui Tan]{\newline Department of Mathematics, \ Rice
  University, 6100 Main St., Houston, TX 77005, USA}
\email{ctan@rice.edu}

\newtheorem{theorem}{Theorem}[section]
\newtheorem{lemma}{Lemma}[section]

\newtheorem{remark}{Remark}[section]

\newcommand{\R}{\mathbb R}
\newcommand{\T}{\mathbb T}
\newcommand{\pa}{\partial}

\begin{document}
\allowdisplaybreaks

\date{\today}


\begin{abstract}
The Euler-Poisson-Alignment (EPA) system appears in mathematical biology and is used to model, in a hydrodynamic limit, a set agents interacting through
mutual attraction/repulsion as well as alignment forces. We consider one-dimensional EPA system with a class of singular alignment terms as well as natural
attraction/repulsion terms. The singularity of the alignment kernel produces an interesting effect regularizing the solutions of the equation and
leading to global regularity for wide range of initial data. This was recently observed in \cite{do2017global}. Our goal in this paper is to generalize the
result and to incorporate the attractive/repulsive potential. We prove that global regularity persists for these more general models.
\end{abstract}

\maketitle \centerline{\date}
\tableofcontents

\section{Introduction and statement of main results}\label{sec1}
We consider the following 1D system of pressureless Euler equations
with nonlocal interaction forces
\begin{align}
&\pa_t \rho + \pa_x (\rho u) = 0,\quad x \in \R, \quad t > 0, \label{eq:mainrho} \\
&\,\,\,\, \pa_t u + u\pa_x u =  \int_{\R} \psi(x-y)(u(y,t) - u(x,t))\rho(y,t)dy
- \pa_x K \star \rho,\label{eq:mainu}
\end{align}
subject to initial density and velocity
\begin{equation}\label{eq:maininitial}
(\rho(\cdot,t),u(\cdot,t))|_{t=0}= (\rho_0, u_0).
\end{equation}

The term on the right hand side of \eqref{eq:mainu} consists of two
parts: an \emph{alignment} interaction with communication weight
$\psi$, and an \emph{attraction-repulsion} interaction through a
potential $K$.

\subsection{Self-organized dynamics with three-zone
  interactions}\label{subsec:intro:3zone}
System \eqref{eq:mainrho}-\eqref{eq:mainu} arises from many contexts
in mathematical physics and biology. In particular, it serves as a
macroscopic system in modeling collective behaviors of complex
biological systems. The corresponding agent-based model has the form
\begin{equation}\label{eq:ABM}
\dot{x}_i=v_i, \quad
m\dot{v}_i=\frac{1}{N}\sum_{j=1}^N\psi(x_i-x_j)(v_j-v_i)-
\frac{1}{N}\sum_{j=1}^N\nabla_{x_i}K(x_i-x_j),
\end{equation}
where $(x_i, v_i)_{i=1}^N$ represent the position and velocity of agent
$i$. The dynamics is governed by Newton's second law, with the
interaction force modeled under a celebrated ``three-zone'' framework
proposed in \cite{reynolds1987flocks}, including
long-range attraction, short-range repulsion, and mid-range alignment.

The first part of the force describes the alignment interaction, where
$\psi$ characterizes the strength of the velocity alignment between
two agents. Naturally, it is a decreasing function of the distance
between agents. Such alignment force has been proposed by Cucker and Smale
in \cite{cucker2007emergent}. The corresponding dynamics enjoys the
flocking property \cite{ha2009simple}, which is a common phenomenon
observed in animal groups.

The second part of the force represents the attraction-repulsion
interaction. The sign of the force $-\nabla K$ determines whether the
interaction is attractive or repulsive. This type of potential driven
interaction force is widely considered in many physical and biological
models, e.g. \cite{d2006self,mogilner1999non}.

Starting from the agent-based model \eqref{eq:ABM}, one can derive a
kinetic representation of the system that describes the mean-field
behavior as $N\to\infty$, see \cite{carrillo2010asymptotic, ha2008particle,
tan2017discontinuous}. Then, a variety of hydrodynamics limits can be
obtained that capture the macroscopic behaviors in different regimes
\cite{fetecau2016first, karper2015hydrodynamic, poyato2017euler}. In
particular, if we consider the mono-kinetic regime, the corresponding
macroscopic system becomes \eqref{eq:mainrho}-\eqref{eq:mainu}.

\subsection{Global regularity versus finite time blowup}
We are interested in the global existence and regularity for the
solution of the system \eqref{eq:mainrho}-\eqref{eq:mainu}.

Let us start with the case with no interaction forces, namely
$\psi=K\equiv0$. The system can be recognized as the pressureless Euler system. In
particular, \eqref{eq:mainu} becomes the classical inviscid Burgers
equation, where smooth data forms shock discontinuity in finite
time due to nonlinear convection $u\pa_xu$.
Together with \eqref{eq:mainrho}, it is well-known that the
solution generates singular shocks in finite time:
$\rho(x,t)\to\infty$ at the position and time when shock occurs.

With alignment force $\psi\geq0$ and $K\equiv0$, the system is called
the \emph{Euler-Alignment system}. When $\psi$ is Lipschitz, the system has been
studied in \cite{carrillo2016critical, tadmor2014critical}, where it is
discovered that the alignment force tends to regularize the solution
and prevent finite time blowup, but only for some initial data. This is so
called \emph{critical threshold phenomenon}: for subcritical initial
data, the alignment force beats the nonlinear convection, and the
solution is globally regular; while for supercritical initial data,
the convection wins and the solution admits a finite time blowup.

Another interesting and natural setting is when $\psi$ is singular, taking the
form
\begin{equation}\label{eq:singularpsi}
\psi(x)=\frac{c_\alpha}{|x|^{1+\alpha}},\quad\alpha>0,
\end{equation}
with $c_\alpha$ be a positive constant. The range $0 <\alpha \leq 2$ is most natural,
and the case $0 < \alpha <1$ is most interesting for the reasons explained later in this sub-section.
The Euler-Alignment system corresponding to the choice \eqref{eq:singularpsi} is studied in
\cite{do2017global} for the periodic case.

Without loss of generality, we set the scale and let $\T=[-1/2,1/2]$ be the periodic domain
of size 1. The singular alignment force can be equivalently expressed as
\begin{equation}\label{eq:singularforce}
\int_\T\psi_\alpha(y)(u(x+y,t)-u(x,t))\rho(x+y,t)dy,
\end{equation}
with the periodic influence function $\psi_\alpha$ defined as
\begin{equation}\label{eq:psip}
\psi_\alpha(x)=\sum_{m\in\mathbb{Z}}\frac{c_\alpha}{|x+m|^{1+\alpha}},
\quad\forall~x\in\T\backslash\{0\}.
\end{equation}
Clearly, $\psi_\alpha$ is singular at $x=0$. Moreover, it has a positive
lower bound
\begin{equation}\label{eq:psim}
\psi_m=\psi_m(\alpha):=\min_{x\in\T}\psi_\alpha(x)=\psi_\alpha\left(\frac{1}{2}\right)>0.
\end{equation}

This leads to the following fractional Euler-Alignment system
\begin{align}
&\pa_t \rho + \pa_x (\rho u) = 0,\qquad x\in\T,\label{eq:fEArho}\\
&\pa_t u + u\pa_x u=\int_\T\psi_\alpha(y)(u(x+y,t)-u(x,t))\rho(x+y,t)dy.
\label{eq:fEAu}
\end{align}
It is shown in \cite{do2017global} that system
\eqref{eq:fEArho}-\eqref{eq:fEAu} has a global smooth solution for all
smooth initial data with $\rho_0>0$. This result is most interesting for the case $0<\alpha<1:$
if we set $\rho \equiv 1$ in \eqref{eq:fEAu}, we get Burgers equation with fractional dissipation.
It is well-known that in this case, there exist initial data leading to finite time blow up (when
$0<\alpha <1;$ the $1 \leq \alpha \leq 2$ range leads to global regularity). However, it turns out
that in the nonlinear
disspation/allignment case described by \eqref{eq:fEAu},
the singularity in the influence function and density modulation
dominate the nonlinaer convection, for all initial data. This also contrasts with the case of Lischitz regular influence function
$\psi,$ where one has critical threshold in the phase space separating initial data leading to finite time blow up
and to global regularity.

\subsection{Euler-Poisson-Alignment system}\label{subsec:intro:EAP}
Now, we take into account the attraction-repulsion force, namely
$K\not\equiv0$. We shall begin with a particular potential
\begin{equation}\label{eq:NewtonianR}
\mathcal{N}(x)=\frac{k|x|}{2}.
\end{equation}
The potential is the \emph{1D Newtonian potential}, and it is the
kernel for the 1D Poisson equation, namely
\[\pa_x^2\mathcal{N}\star\rho=k\rho.\]
When $k>0$, the Newtonian force
$\pa_x^2 \mathcal{N}\star\rho$ is attractive, and when $k<0$, the Newtonian force
is repulsive.
We call the corresponding system Euler-Poisson-Alignment (EPA) system. It
has the form
\begin{align}
&\pa_t \rho + \pa_x (\rho u) = 0,\label{eq:EPArho}\\
&\pa_t u + u \pa_x u = -\pa_x \phi + \int_{\R} \psi(x-y)(u(y,t) - u(x,t))
\rho(y,t) dy, \label{eq:EPAu}
\end{align}
where the stream function $\phi=\mathcal{N}\star\rho$ satisfies the Poisson equation
\begin{equation}\label{eq:EPAphi}
\pa_x^2\phi =k\rho.
\end{equation}

When there is no alignment force $\psi\equiv0$, the system coincides
with the 1D pressureless Euler-Poisson equation, which has been
extensively studied in \cite{engelberg2001critical}. The result is as follows:
when $k>0$, the attraction force together with convection
drives the solution of Euler-Poisson equation to a finite time
blowup for all smooth initial data; when $k<0$, the repulsive force
competes with the convection, and there exists a critical
threshold on initial conditions which separates global regularity and
finite time blowup.

The EPA system \eqref{eq:EPArho}-\eqref{eq:EPAphi} is studied in
\cite{carrillo2016critical}, in the case when $\psi$ is Lipschitz .
When $k<0$, a larger subcritical region of initial
data is obtained that ensures global regularity. This implies that the
alignment force helps repulsive potential to compete with the
convection. However, it is also shown that when $k>0$, the alignment
force is too weak to compete with convection and attractive potential,
so all smooth initial data lead to finite time blow up.

Our first result concerns EPA system with singular alignment force,
where the influence function has the form \eqref{eq:singularpsi}. The
main goal is to understand whether the singular alignment can
still regularize the solution when the Newtonian force is present.

We shall study the system in the periodic setting.
The 1D periodic Newtonian potential reads
\begin{equation}\label{eq:Newtonian}
\mathcal{N}(x)=-\frac{k}{2}\left(\frac{1}{2}-|x|\right)^2,\quad\forall~x\in\T.
\end{equation}
It is the kernel of the Poisson equation with background, namely
\[\partial_x^2(\mathcal{N}\ast\rho)=k(\rho-\bar{\rho}),\]
where $\bar{\rho}$ is the average density
\begin{equation}\label{eq:rhobar}
\bar{\rho}=\frac{1}{|\T|}\int_\T\rho(x,t)dx=\int_\T\rho_0(x)dx.
\end{equation}
Note that $\bar{\rho}$ is conserved in time due to conservation of mass by evolution.
The stream function $\phi$ in \eqref{eq:EPAu} satisfies the Poisson
equation with constant background
\begin{equation}\label{eq:EPAphiP}
\pa_x^2\phi =k(\rho-\bar{\rho}).
\end{equation}

The presence of the background $\bar{\rho}$ could change the
behavior of the solution.
For Euler-Poisson equation in periodic domain, namely
\eqref{eq:EPArho}-\eqref{eq:EPAu},\eqref{eq:EPAphiP} with $\psi\equiv0$, it is pointed out in
\cite{engelberg2001critical} that the
background has the tendency to balance both the
convection and attractive forces. So for the attractive case $k>0$,
instead of finite time blowup for all initial data, a critical
threshold is obtained.

Though similar techniques in \cite{carrillo2016critical}, one can
derive critical thresholds for EPA system
\eqref{eq:EPArho}-\eqref{eq:EPAu},\eqref{eq:EPAphiP} with bounded Lipschitz
influence function $\psi$.

The EPA system with singular alignment force \eqref{eq:singularforce} and potential
\eqref{eq:Newtonian} reads
\begin{align}
&\pa_t \rho + \pa_x (\rho u) = 0,\label{eq:EPASrho}\\
&\pa_t u + u \pa_x u = -\pa_x \phi + \int_\T\psi_\alpha(y)(u(x+y,t)-u(x,t))\rho(x+y,t)dy,  \,\,\pa_x^2\phi =k(\rho-\bar{\rho}).\label{eq:EPASu}
\end{align}

The following theorem shows that the singular alignment force
dominates the Poisson force, and global regularity is obtained for all
initial data.
\begin{theorem}\label{thm:EAP}
For $\alpha\in(0,1)$, the fractional EPA system
\eqref{eq:EPASrho}-\eqref{eq:EPASu} with smooth periodic initial data
$(\rho_0, u_0)$ such that $\rho_0>0$ has a unique smooth
solution.
\end{theorem}
\begin{remark} The proof can be easily extended to the range $\alpha \geq 1$ with more straightforward arguments
for $\alpha >1;$ see also \cite{shvydkoy2017eulerian} for a different approach.
We focus on the $0 < \alpha <1$ case in the rest of the paper.
\end{remark}

We note that the proof of global regularity in \cite{do2017global} is based, in particular, on rather precise algebraic
structures that we will discuss below. Even though the interaction force we are adding is formally sub-critical,
it is far from obvious that the fairly intricate arguments of \cite{do2017global} survive such perturbation.

\subsection{Euler dynamics with general three-zone interactions}
The results on EPA system can be extended to systems with more
general interaction forces.

In \cite{carrillo2016critical}, critical thresholds are obtained for
the system \eqref{eq:mainrho}-\eqref{eq:mainu}, with
Lipschitz influence function $\psi$, and regular potential $K\in W^{2,\infty}$.

In this paper, we will also consider the case of more general singular influence function $\psi$.
More precisely, we assume that $\psi\geq\psi_m>0$, and  can
be decomposed into
two parts
\begin{equation}\label{eq:psidecomp}
\psi=c\psi_\alpha+\psi_L,
\end{equation}
where $c>0$, $\psi_\alpha$ is defined in \eqref{eq:psip}, and $\psi_L$ is a bounded
Lipschitz function.

\begin{theorem}\label{thm:3zone}
Consider system
\eqref{eq:mainrho}-\eqref{eq:mainu} in the periodic setup
\begin{align}
&\pa_t \rho + \pa_x (\rho u) = 0,\quad x \in \T, \quad t > 0, \label{eq:mainrhop} \\
&\,\,\,\, \pa_t u + u\pa_x u =  \int_\T\psi(y)(u(x+y,t)-u(x,t))\rho(x+y,t)dy
- \pa_x K \star \rho,\label{eq:mainup}
\end{align}
with smooth initial data $(\rho_0, u_0)$ such that $\rho_0>0$. Assume
$\psi$ is singular in the sense of \eqref{eq:psidecomp}, and $K$ is
a linear combination of Newtonian potential \eqref{eq:Newtonian}
and regular $W^{2,\infty}(\T)$ potential.

Then, the system has a unique global smooth solution.
\end{theorem}

We summarize the global behaviors of Euler equations with nonlocal
interaction \eqref{eq:mainrho}-\eqref{eq:mainu} under different
choices of interaction forces.
\begin{table}[h]
{\def\arraystretch{1.5}
\begin{tabular}{|c|c|c|c|l|}
\hline
Potential&Alignment
&Name&Domain&Behaviors\\ \hline
No&No&Euler& $\R$ or $\T$&Finite time blow up \\ \cline{2-5}
&Lipshitz&Euler-Alignment&$\R$ or $\T$& Critical threshold \cite{carrillo2016critical,tadmor2014critical}\\ \cline{2-5}
&Singular&Fractional EA&$\T$& Global regularity \cite{do2017global}\\ \hline
Newtonian&No&Euler-Poisson&$\R$& Finite time blow up \cite{engelberg2001critical}
\\ \cline{4-5}
&&&$\T$& Critical threshold \cite{engelberg2001critical}
\\ \cline{2-5}
&Lipschitz&EPA&$\R$& Finite time blow up (attractive)\\
&&&& Critical threshold (repulsive) \cite{carrillo2016critical}\\ \cline{4-5}
&&&$\T$& Critical threshold\\ \cline{2-5}
&Singular&Fractional EPA&$\T$& Global regularity (Theorem \ref{thm:EAP})\\ \hline
General&Lipshitz&Euler-3Zone& $\R$ or $\T$& Critical thresholds \cite{carrillo2016critical}\\ \cline{2-5}
&Singular&Singular 3Zone& $\T$&  Global regularity (Theorem \ref{thm:3zone})\\ \hline
\end{tabular}
}
\end{table}

\section{Euler-Poisson-Alignment system}\label{sec:EAP}
In this section, we consider Euler-Poisson-Alignment system
\eqref{eq:EPASrho}-\eqref{eq:EPASu}
with singular alignment force \eqref{eq:singularforce}.

Following the idea in \cite{do2017global}, we let
\begin{equation}\label{eq:G}
G=\pa_xu-\Lambda^\alpha\rho,
\end{equation}
and calculate the dynamics of $G$ using \eqref{eq:EPASrho} and
\eqref{eq:EPASu}:
\begin{align*}
\pa_tG=&~\pa_t\pa_xu-\pa_t\Lambda^\alpha\rho
=-\pa_x(u\pa_xu)-k(\rho-\bar{\rho})+\pa_x\left(-\Lambda^\alpha(\rho
         u)+u\Lambda^\alpha\rho\right)+\Lambda^\alpha\pa_x(\rho u)\\
=&-u\pa_x(\pa_xu-\Lambda^\alpha\rho)-\pa_xu(\pa_xu-\Lambda^\alpha\rho)-k(\rho-\bar{\rho})
=-\pa_x(Gu)-k(\rho-\bar{\rho}).
\end{align*}
So, we can rewrite the dynamics in terms of $(\rho,G)$ as
\begin{align}
&\pa_t \rho + \pa_x (\rho u) = 0,\label{eq:EPASrho2}\\
&\pa_t G + \pa_x (Gu) = -k(\rho-\bar{\rho}),\label{eq:EPASG}\\
&\pa_xu=\Lambda^\alpha\rho+G.\label{eq:EPASux}
\end{align}

The velocity $u$ can be recovered as
\begin{equation}\label{eq:urec}
u(x,t)=\Lambda^\alpha\partial_x^{-1}(\rho(x,t)-\bar{\rho})+\partial_x^{-1}G(x,t)+I_0(t),
\end{equation}
where $I_0$ can be determined by conservation of momentum.
\begin{equation}\label{eq:mom}
\int_\T\rho(x,t) u(x,t)dx=\int_\T\rho_0(x) u_0(x)dx.
\end{equation}
See \cite{do2017global} for detailed discussion.

\subsection{A priori bounds}
We first show an upper and lower bounds on density $\rho$ for all
finite times. For $k=0$, a uniform in time
bound is obtained in \cite{do2017global}.
With the Newtonian potential, especially when $k>0$, the
attractive force definitely helps density concentration. Hence, the upper bound
on $\rho$ can be expected to grow in time. However, the bound we obtain in this section indicates
that there is no finite time singular concentration on density, thanks to the
singular alignment force.

Let $F=G/\rho$. We can rewrite \eqref{eq:EPASrho2} as
\begin{equation}\label{eq:rhody}
(\pa_t+u\pa_x)\rho=-\rho\Lambda^\alpha\rho-\rho^2F.
\end{equation}

The first step is to obtain a bound on $F$. We calculate
\[\pa_tF=\frac{\rho\pa_tG-G\pa_t\rho}{\rho^2}
=\frac{\rho(-\pa_x(Gu)-k(\rho-\bar{\rho}))-G(-\pa_x(\rho u))}{\rho^2}
=-u\pa_xF-\frac{k(\rho-\bar{\rho})}{\rho}.
\]
This implies that
\begin{equation}\label{eq:Fcha}
(\pa_t+u\pa_x) F=-k\left(1-\frac{\bar{\rho}}{\rho}\right).
\end{equation}
Denote $X(x,t)$ the trajectory of the characterstic path starting
at $x$, namely
\begin{equation}\label{eq:path}
\frac{d}{dt}X(x,t)=u(X(x,t),t),\quad X(x,0)=x.
\end{equation}
Then, we can solve for $F$ along the characteristic path
\begin{equation}\label{eq:F}
F(X(x,t),t)=F_0(x)-kt+\int_0^t\frac{k\bar{\rho}}{\rho(X(x,s),s)}ds.
\end{equation}
Define $\rho_m(t)$ as the lower bound of $\rho$ on time interval $[0,t]$
\begin{equation}\label{eq:rhom}
\rho_m(t)=\min_{s\in[0,T]}\min_{x\in\T}\rho(x,s).
\end{equation}
Then, we get a bound on $F$ from \eqref{eq:F}:
\begin{equation}\label{eq:Fbound}
\|F(\cdot,t)\|_{L^\infty}\leq\|F_0\|_{L^\infty}+|k|t+|k|\bar{\rho}\int_0^t\frac{1}{\rho_m(s)}ds.
\end{equation}
Therefore, in order to control $F$ in $L^\infty$,  we need a lower bound estimate on
the density.

\begin{theorem}[Lower bound on density]\label{thm:lower}
Let $(\rho, u)$ be a strong solution to EPA system
\eqref{eq:EPASrho}\eqref{eq:EPASu} with smooth periodic initial
conditions $(\rho_0, u_0)$ such that $\rho_m(0)>0$.
Then, there exist two positive constants $A_m$ and $C_m$,
depending only on the initial conditions, such that
for any $t\geq 0$,
\begin{equation}\label{eq:lowerbound}
\rho_m(t)\geq C_me^{-A_mt}.
\end{equation}
\end{theorem}

\begin{proof}
We depart from \eqref{eq:rhody} and estimate $\Lambda^\alpha\rho$ and $F$.
For a fixed time $t$, denote $\underline{x}$ be a point where $\rho$
attains its minimum. Note that $\underline{x}$ depends on $t$ and it
is not necessarily unique. The estimates below apply at any such point. We have
\begin{equation}\label{eq:lb1}
\begin{aligned}
-\Lambda^\alpha\rho(\underline{x},t)=&~c_\alpha\int_{-\infty}^\infty
\frac{\rho(\underline{x}+y,t)-\rho(\underline{x},t)}{|y|^{1+\alpha}}dy
=\int_\T\psi_\alpha(y)\big(\rho(\underline{x}+y,t)-\rho(\underline{x},t)\big)dy
\\
\geq&~
\psi_m\int_\T\big(\rho(\underline{x}+y,t)-\rho(\underline{x},t)\big)dy
=\psi_m\big(\bar{\rho}-\rho(\underline{x},t)\big).
\end{aligned}
\end{equation}
Here, we recall that $\psi_m$ is the positive lower bound of
$\psi_\alpha$ defined in \eqref{eq:psim}.

Combining \eqref{eq:Fbound} and \eqref{eq:lb1}, we obtain
\begin{equation}\label{eq:rhombound1}
\pa_t\rho(\underline{x},t)\geq~
\big(\psi_m\bar{\rho}\big)\rho(\underline{x},t)-\left[\psi_m+
\|F_0\|_{L^\infty}+|k|t+|k|\bar{\rho}\int_0^t\frac{1}{\rho_m(s)}ds\right]\rho(\underline{x},t)^2.
\end{equation}

We prove \eqref{eq:lowerbound} by contradiction.
For $t=0$, the bound \eqref{eq:lowerbound} holds if we let
$C_m\leq\rho_m(0)$. Suppose \eqref{eq:lowerbound} does not hold for
all $t\geq0$. Then, there exists a finite time $t_0>0$ so that the
inequality is violated for the first time at $t=t_0+$.
Pick any $\underline{x}=\underline{x}(t_0)$. Due to continuity of
$\rho$, we know
\begin{equation}\label{eq:rhomb}
\rho_m(t_0)=\rho(\underline{x},t_0)=C_me^{-A_mt_0}.
\end{equation}
Plug in \eqref{eq:rhomb} to \eqref{eq:rhombound1} and use the fact
that \eqref{eq:lowerbound} holds for all $t\in[0,t_0]$. We get
\begin{align*}
\pa_t\rho(\underline{x},t_0)\geq&~\rho_m(t_0)\left[
\big(\psi_m\bar{\rho}\big)-\left(\psi_m+
\|F_0\|_{L^\infty}+|k|t+|k|\bar{\rho}\int_0^{t_0}\frac{1}{\rho_m(s)}ds\right)\rho_m(t_0)
\right]\\
\geq&~\rho_m(t_0)\left[
\big(\psi_m\bar{\rho}\big)-\left(\psi_m+
\|F_0\|_{L^\infty}+|k|t_0+\frac{|k|\bar{\rho}}{A_mC_m}(e^{A_mt_0}-1)\right)C_me^{-A_mt_0}
\right]\\
\geq&~\rho_m(t_0)\left[
\left(\psi_m\bar{\rho}-\frac{|k|\bar{\rho}}{A_m}\right)-\left(\psi_m+
\|F_0\|_{L^\infty}+|k|t_0-\frac{|k|\bar{\rho}}{A_mC_m}\right)C_me^{-A_mt_0}
\right]\\
\geq&~\rho_m(t_0)\left[
\left(\psi_m\bar{\rho}-\frac{|k|\bar{\rho}}{A_m}-\frac{|k|C_m}{eA_m}\right)
+\left(\frac{|k|\bar{\rho}}{A_m}-C_m(\psi_m+\|F_0\|_{L^\infty})\right)e^{-A_mt_0}
\right].
\end{align*}
The right hand side is positive if we pick $A_m$ large enough and
$C_m$ small enough. For instance, we can pick
\begin{equation}\label{eq:AC}
A_m=\frac{|k|}{\psi_m}(1+\epsilon),\quad
C_m=\min\{\rho_m(0), \epsilon e\bar{\rho}\},
\end{equation}
for any $\epsilon\in(0,\epsilon_*)$, where $\epsilon_*=\frac{1}{2}\left(\sqrt{1+\frac{4\psi_m}{e(\psi_m+\|F_0\|_{L^\infty})}}-1\right)$.
With this choice of $A_m$ and $C_m$, we get
$\pa_t\rho(\underline{x},t_0)> 0$.

Now we obtain that $\rho(\underline{x}) < C_m e^{-A_m t_0} < C_m e^{-A_m t}$ for some $t < t_0.$
This contradicts our choice of $t_0.$

\end{proof}

\begin{remark}
The bound \eqref{eq:lowerbound} with decay rate \eqref{eq:AC}
is not necessarily sharp, but is enough for our purpose, as it
eliminates the possibility of finite time creation of vacuum.
One important observation is that for $k=0$, we get $A_m=0$.
In this case, the lower bound is uniform in time.
\end{remark}

Applying the lower bound \eqref{eq:lowerbound} to \eqref{eq:Fbound}, we
immediately derive a bound on $F$
\begin{equation}\label{eq:Fbound2}
\|F(\cdot,t)\|_{L^\infty}\leq \|F_0\|_{L^\infty} +
|k|t+\frac{|k|\bar{\rho}}{A_mC_m}e^{A_mt}
=:F_M(t).
\end{equation}

Now, we are ready to obtain an upper bound on density $\rho$.

\begin{theorem}[Upper bound on density]\label{thm:upper}
Let $(\rho, u)$ be a strong solution to EPA system
\eqref{eq:EPASrho}\eqref{eq:EPASu} with smooth periodic initial
conditions $(\rho_0, u_0)$ such that $\rho_m(0)>0$.
Then, there exist two positive constants $A_M$ and $C_M$,
depending only on the initial conditions, such that
for any $t\geq 0$ and $x\in\T$,
\begin{equation}\label{eq:upperbound}
\rho(x,t)\leq \rho_M(t):=C_Me^{A_Mt}.
\end{equation}
\end{theorem}

\begin{proof}
We again depart from \eqref{eq:rhody} and start with a lower bound
estimate on $\Lambda^\alpha\rho$.
For a fixed time $t$, denote $\bar{x}$ be a point where $\rho$
attains its maximum. Applying
nonlinear maximum principle by Constantin and Vicol
\cite{constantin2012nonlinear}, one can estimate
\begin{equation}\label{eq:CVbound}
\Lambda^\alpha\rho(\bar{x},t)\geq C_1\rho(\bar{x},t)^{1+\alpha},
\end{equation}
if $\rho(\bar{x},t)\geq 3\bar{\rho}$. The constant $C_1$ only depends
on initial conditions. One can consult \cite{do2017global} for more
details of the estimate.

Plugging the estimates \eqref{eq:Fbound2} and \eqref{eq:CVbound} into
\eqref{eq:rhody}, we obtain
\begin{equation}\label{eq:upperbound1}
\pa_t\rho(\bar{x},t)\leq -C_1\rho(\bar{x},t)^{2+\alpha}+
F_M(t)\rho(\bar{x},t)^2.
\end{equation}
It follows that $\pa_t\rho(\bar{x},t)<0$ if
$\rho(\bar{x},t)>(F_M/C_1)^{1/\alpha}$. Therefore,
\begin{equation}\label{eq:rhomaxb}
\rho(x,t)\leq\rho(\bar{x},t)\leq\max\left\{\|\rho_0\|_{L^\infty},
  3\bar{\rho}, \left(\frac{F_M(t)}{C_1}\right)^{1/\alpha}\right\},
\end{equation}
and \eqref{eq:upperbound} holds with
\[A_M=\frac{A_m}{\alpha},\quad
C_M=\max\left\{\max_{x\in\T}\rho_0(x), ~ 3\bar{\rho}, ~
\left[\frac{1}{C_1}\left(\|F_0\|_{L^\infty}+\frac{|k|}{eA_m}+\frac{|k|\bar{\rho}}{A_mC_m}\right)\right]^{1/\alpha}\right\}.\]
\end{proof}

\subsection{Local wellposedness}
With the apriori bounds, we state a local wellposedness result for the fractional
EPA system~(\ref{eq:EPASrho})-(\ref{eq:EPASu}), as well as a
Beale-Kato-Majda type necessary and sufficient condition to guarantee global
wellposedness. The local wellposedness theory has been presented in detail in
\cite{do2017global} for fractional Euler-Alignment system.
We will show that presence of the Poisson force does not seriously affect the
argument, no matter whether it is attractive or repulsive. We will only sketch the
proof, indicating changes necessary.

\begin{theorem}[Local wellposedness]\label{thm:local}
Consider EPA system~(\ref{eq:EPASrho})-(\ref{eq:EPASu}) with initial
conditions $\rho_0$ and $u_0$ that satisfy
\begin{equation}\label{eq:smoothinit}
\rho_0\in H^s(\T),\quad \min_{x\in\T}\rho_0(x)>0,\quad
\partial_xu_0-\Lambda^\alpha\rho_0\in H^{s-\frac{\alpha}{2}}(\T),
\end{equation}
with a sufficiently large even integer $s>0$. Then, there exists $T_0>0$
such that the EPA system
has a unique strong solution $\rho(x,t), u(x,t)$ on $[0,T_0]$, with
\begin{equation}\label{eq:smoothsol}
\rho\in C([0,T_0], H^s(\T))\cap L^2([0,T_0], H^{s+\frac{\alpha}{2}}(\T)),\quad
u\in C([0,T_0], H^{s+1-\alpha}(\T)).
\end{equation}
Moreover, a necessary and sufficient condition for the solution to
exist on a time interval~$[0,T]$ is
\begin{equation}\label{eq:BKM}
\int_0^T\|\partial_x\rho(\cdot,t)\|_{L^\infty}^2 dt<\infty.
\end{equation}
\end{theorem}
\begin{proof}
We follow the proof in \cite{do2017global} and rewrite the equations
\eqref{eq:EPASrho2} and \eqref{eq:EPASG} in terms of
$(\theta, G)$ where $\theta=\rho-\bar{\rho}$.
\begin{align}
\pa_t\theta+\pa_x(\theta u)&=-\bar{\rho}\pa_xu,\label{eq:EAPtheta}\\
\pa_tG+\pa_x(Gu)&=-k\theta,\label{eq:EAPG2}
\end{align}
The velocity $u$ is defined in \eqref{eq:urec}.

Given any $T>0$, we will obtain a differential inequality on
\begin{equation}\label{eq:Y}
Y(t):=1+\|\theta(\cdot,t)\|_{H^s}^2+\|G(\cdot,t)\|_{H^{s-\frac{\alpha}{2}}}^2,
\end{equation}
for all $t\in[0,T]$.

Through a commutator estimate \cite[equation (3.23)]{do2017global}, one can get
\begin{equation}\label{eq:thetaHs}
\frac{1}{2}\frac{d}{dt}\|\theta\|_{H^s}^2\leq
C\left(1+\frac{1}{\rho_m}\right)
(1+\|\partial_x\theta\|_{L^\infty}^2+\|G\|_{L^\infty})
Y(t)-
\frac{\rho_m}{3}\|\theta\|_{H^{s+\frac{\alpha}{2}}}^2,
\end{equation}
where $\rho_m(t)$ has a positive lower bound for $t\in[0,T]$ due to
Theorem \ref{thm:lower}. Also, $\|G(\cdot,t)\|_{L^\infty}$ is bounded
for $t\in[0,T]$ as $G=F\rho$ and both $F$ and $\rho$ are bounded, see
\eqref{eq:Fbound2} and \eqref{eq:upperbound} respectively.

We also compute
\begin{equation}\label{eq:GHs}
\frac{1}{2}\frac{d}{dt}\|G\|_{\dot{H}^{s-\frac{\alpha}{2}}}^2
=-\int_\T(\Lambda^{s-\frac{\alpha}{2}}
   G)\cdot(\Lambda^{s-\frac{\alpha}{2}}\partial_x (Gu))dx
-k\int_\T(\Lambda^{s-\frac{\alpha}{2}}
   G)\cdot (\Lambda^{s-\frac{\alpha}{2}}\theta)dx=I+II.
\end{equation}
The first term can be controlled by the following estimate
\cite[equation (3.25)]{do2017global}
\begin{equation}\label{eq:GHsI}
|I|\leq
\frac{\rho_m}{6}\|\theta\|_{H^{s+\frac{\alpha}{2}}}^2 +C
\left(1+\frac{1}{\rho_m}\|G\|_{L^\infty}^2+\|\partial_x\theta\|_{L^\infty}+\|G\|_{L^\infty})\right) \|G\|_{H^{s-\frac{\alpha}{2}}}^2.
\end{equation}

The $II$ term encodes the contribution of the attractive-repulsive
potential. We have the following estimate
\begin{equation}\label{eq:GHsII}
|II|\leq
|k|\|G\|_{\dot{H}^{s-\frac{\alpha}{2}}}\|\theta\|_{\dot{H}^{s-\frac{\alpha}{2}}}
\leq C|k|\|G\|_{\dot{H}^{s-\frac{\alpha}{2}}}\|\theta\|_{H^s}\leq C|k|Y(t).
\end{equation}
Combine \eqref{eq:thetaHs}, \eqref{eq:GHsI} and \eqref{eq:GHsII}, we
get
\begin{equation}\label{eq:YHs}
\frac{d}{dt}Y(t)\leq C(1+\|\pa_x\theta(\cdot,t)\|_{L^\infty}^2)Y(t)-\frac{\rho_m(t)}{6}\|\theta\|_{H^{s+\frac{\alpha}{2}}}^2,
\end{equation}
where $C$ is a positive constant which might depend on $T$.

Applying Gronwall's inequality, we get
\begin{equation}\label{eq:gronwall}
Y(t)+\frac{1}{6}\min_{t\in[0,T]}\rho_m(t)\|\theta\|_{L^2([0,T];H^{s+\frac{\alpha}{2}}(\T))}^2\leq
Y(0)\exp\left[C(T)\int_0^T(1+\|\pa_x\theta(\cdot,s)\|_{L^\infty}^2)ds\right],
\end{equation}
for all $t\in[0,T]$.
The right hand side is bounded as long as condition \eqref{eq:BKM} is
satisfied. Therefore,
\[\theta\in C([0,T], H^s(\T))\cap L^2([0,T], H^{s+\frac{\alpha}{2}}(\T)),\quad
G\in C([0,T], H^{s-\frac{\alpha}{2}}(\T)).\]
This directly implies the regularity conditions on $\rho$ in
\eqref{eq:smoothsol}. The regularity conditions on $u$ can also be easily
obtained from \eqref{eq:urec}.
\end{proof}

\subsection{Global wellposedness}\label{sec:global}
In this section, we prove that the Beale-Kato-Majda type condition
\eqref{eq:BKM} holds for any finite time $T$. This will imply global
wellposedness of the fractional EPA system and hence finish the proof
of Theorem \ref{thm:EAP}. Throughout the section, we fix a time $T>0$
(which is arbitrary).

To derive a uniform $L^\infty$ bound on $\pa_x\rho$, we argue that
$\rho(\cdot,t)$ will obey certain modulus of continuity for
$t\in[0,T]$. Such method has been successfully used to obtain global
regularity for 2D quasi-geostrophic equation with critical dissipation
\cite{kiselev2007global}, fractal Burgers equation
\cite{kiselev2008blow}, as well as fractional Euler-Alignment system
\cite{do2017global}. In all these examples, the solution has a certain
scaling invariance property. Unfortunately, such property is not
available for the fractional EPA system
\eqref{eq:EPASrho}-\eqref{eq:EPASu}. We note that the modulus method
has been applied to subcritical perturbations destroying scaling before
(e.g. \cite{kiselev2010global}). The argument in \cite{kiselev2010global}, however, relies on the specific
structure of the perturbation, and cannot be readily ported to other settings.
A novel feature compared to both \cite{kiselev2010global} and \cite{do2017global} will be
dependence of the modulus on time. This feature is linked to the possible decay
of $\rho_m$ and growth of $\|\rho\|_{L^\infty},$ and appears to be an intrinsic
property of the problem.

We use the same family of moduli of continuity as in \cite{do2017global},
\begin{equation}\label{eq:mocf}
\omega(\xi)=\begin{cases}\xi-\xi^{1+{\alpha}/{2}},&0\le\xi<\delta \leq 1\\
\gamma\log(\xi/\delta)+\delta-\delta^{1+\alpha/2},&\xi\geq\delta,
\end{cases}
\end{equation}
where $\gamma, \delta$ are small constants to be determined. Set
$\omega_B(\xi)=\omega(B\xi)$, where $B$ is a large constant to be
determined as well. Due to lack of scaling invariance, we
will work directly on $\omega_B$.
\begin{equation}\label{eq:omegaB}
\omega_B(\xi)=\begin{cases}B\xi-(B\xi)^{1+{\alpha}/{2}},&0\le\xi<B^{-1}\delta\\
\gamma\log\frac{B\xi}{\delta}+\delta-\delta^{1+\alpha/2},&\xi\geq B^{-1}\delta,
\end{cases}
\end{equation}

We say that a function $f$ obeys modulus of continuity $\omega$ if
\begin{equation}\label{eq:defmod}
|f(x)-f(y)| < \omega(|x-y|),\quad\forall~x,y\in\T.
\end{equation}

Our plan is to find a $\omega_B$ such that $\rho(\cdot,t)$ obeys
$\omega_B$ for all $t\in[0,T]$. To construct $\omega_B$, we will first
choose $\delta$ and $\gamma$ which
depend on initial conditions and $T$, but not on $B$.
Then, we will choose $B$ that depend on $T,\delta,\gamma$ as well as initial conditions.

First, we would like to make sure that $\rho_0$ obeys $\omega_B$.

\begin{lemma}\label{lem:initmod}
Let $\rho_0\in C^1(\T)$. Then, $\rho_0$ obeys $\omega_B$ if
\begin{equation}\label{eq:deltagammaBinit}
\delta<\frac{2\|\rho_0\|_{L^\infty}}{\|\pa_x\rho_0\|_{L^\infty}},\quad
B>\frac{\delta\|\pa_x\rho_0\|_{L^\infty}}{2\|\rho_0\|_{L^\infty}}\exp\left(\frac{2\|\rho_0\|_{L^\infty}}{\gamma}\right).
\end{equation}
\end{lemma}
\begin{proof}
We start with an elementary inequality
\begin{equation}\label{eq:rhoele}
|\rho_0(x)-\rho_0(y)|\leq\min\{2\|\rho_0\|_{L^\infty},
\|\pa_x\rho_0\|_{L^\infty}|x-y|\}.
\end{equation}
As $\omega_B$ is concave and monotone increasing, the right hand side
of \eqref{eq:rhoele} is bounded by $\omega_B(|x-y|)$ if
\begin{equation}\label{eq:modinit}
\omega_B\left(\frac{2\|\rho_0\|_{L^\infty}}{\|\pa_x\rho_0\|_{L^\infty}}\right)
>2\|\rho_0\|_{L^\infty}.
\end{equation}
Since $\omega_B(\xi)\to+\infty$ as $B\to+\infty$, \eqref{eq:modinit}
is satisfied by taking $B$ large enough. Indeed, if $\delta$ and $B$ satisfy
\eqref{eq:deltagammaBinit}, then
\[\omega_B\left(\frac{2\|\rho_0\|_{L^\infty}}{\|\pa_x\rho_0\|_{L^\infty}}\right)
>\gamma\log\left(\frac{2B\|\rho_0\|_{L^\infty}}{\delta\|\pa_x\rho_0\|_{L^\infty}}\right)>
2\|\rho_0\|_{L^\infty}.\]
\end{proof}

The following lemma describes the only possible breakthrough scenario
for the modulus.
\begin{lemma}\label{lem:breakthrough}
Suppose $\rho_0$ obeys a modulus of continuity $\omega_B$ as in \eqref{eq:omegaB}.
If the solution~$\rho(x,t)$ violates
$\omega_B$ at some positive time, then there must exist $t_1>0$ and
$x_1\neq y_1$ such that
\begin{equation}\label{eq:breakthrough}
\rho(x_1,t_1)-\rho(y_1,t_1)=\omega_B(|x_1-y_1|),
\hbox{  and $\rho(\cdot,t)$ obeys $\omega_B$ for every $0\leq t<t_1$.}
\end{equation}
\end{lemma}

The main point of the lemma is the existence of two distinct points where the solution touches the
modulus (as opposed to a single point $x$ with $|\nabla \rho(x)| = \omega'_B(0)=B$).
This property is a consequence of $\omega''_B(0) = -\infty;$ see \cite{kiselev2007global} for more
details.

We will show that in the breakthrough scenario as above,
\begin{equation}\label{eq:nobreak}
\pa_t(\rho(x_1,t_1)-\rho(y_1,t_1))<0,\quad\forall~ t_1\in(0,T],
\end{equation}
achieving a contradiction with the choice of time $t_1$ - and thus showing
that the modulus $\omega_B$ cannot be broken.
Together
with Lemma \ref{lem:initmod} this implies that
$\rho(\cdot,t)$ obeys
$\omega_B$ for all $t\in[0,T]$. Therefore,
\begin{equation}\label{eq:rhoxbound}
\|\pa_x\rho(\cdot,t)\|_{L^\infty}\leq\omega_B'(0)=B,\quad\forall~t\in[0,T].
\end{equation}
This proves global regularity of the fractional EPA system and ends the proof of  Theorem \ref{thm:EAP}.

The rest of the section is devoted to proof of \eqref{eq:nobreak}.
We fix $t_1$ and drop the time variable for simplicity.
Let $\xi=|x_1-y_1|$. Then
\begin{equation}\label{eq:rhoxi}
\begin{split}
\partial_t&(\rho(x_1)-\rho(y_1))=
-\partial_x(\rho(x_1)u(x_1))+\partial_x(\rho(y_1)u(y_1))\\
&=-\big(u(x_1)\partial_x\rho(x_1)-u(y_1)\partial_x\rho(y_1)\big)
-\big(\rho(x_1)-\rho(y_1)\big)\partial_xu(x_1)
-\rho(y_1)\big(\partial_xu(x_1)-\partial_xu(y_1)\big)\\
&=I+II+III.
\end{split}
\end{equation}

Decompose $u$ into two parts $u=u_1+u_2$ where
\begin{equation}\label{eq:u1u2}
u_1(x)=\Lambda^\alpha\pa_x^{-1}(\rho(x)-\bar{\rho}),\quad
u_2(x)=\pa_x^{-1}G(x)+I_0.
\end{equation}
Then, we can write \eqref{eq:rhoxi} as
\begin{equation}\label{eq:rhoxi2}
\partial_t(\rho(x_1)-\rho(y_1))=I_1+II_1+III_1+I_2+II_2+III_2,
\end{equation}
where $I_1, II_1, III_1$ represent the contributions from $u_1$, and
$I_2, II_2, III_2$ represent the contribution from $u_2$.

For $I_1, II_1, III_1$, we proceed with an argument parallel to
\cite{do2017global}. Let us recall the result.
The following quantities play a role in the proof:
\[ \Omega(\xi) = c_{1,\alpha} \left( \int_0^\xi \frac{\omega(\eta)}{\eta^\alpha} \,d\eta +
\xi \int_{\xi}^\infty \frac{\omega(\eta)}{\eta^{1+\alpha}}\,d\eta \right); \]
\[ A(\xi) = c_{2,\alpha} \int_{\R} \frac{\omega(\xi) - \omega(|\xi-\eta|)}{|\eta|^{1+\alpha}}\,d\eta; \]
{\small\[ D(\xi) = c_{3,\alpha} \left( \int_0^{\xi/2} \frac{2\omega(\xi) - \omega(\xi+2\eta)-\omega(\xi-2\eta)}{\eta^{1+\alpha}}\,d\eta
+\int_{\xi/2}^\infty \frac{2\omega(\xi) - \omega(\xi+2\eta)+\omega(2\eta-\xi)}{\eta^{1+\alpha}}\,d\eta \right). \]}

\begin{lemma}[{\cite[Lemma 4.4 and 4.5]{do2017global}}]\label{lem:homopart1}
Let $\rho(\cdot,t)$ obey the modulus of continuity $\omega$ as in
\eqref{eq:mocf} for $0\leq t<t_1\leq T$, and let $x_1$, $y_1$
be the breakthrough points at the first breakthrough time $t_1$.
Suppose $\delta$ and $\gamma$ are small constants such that
\begin{equation}\label{eq:deltagamma}
\delta<1,\quad\gamma\leq\frac{\delta-\delta^{1+\alpha/2}}{2\log2}.
\end{equation}
Then, there exist positive constants $C_I$, $C_{II}$ and $C_{III}$, which may
only depend on $\alpha$, such that
\begin{align}
|I_1|\leq&~\omega'(\xi)\Omega(\xi),\quad\text{where }
\Omega(\xi)\leq
\begin{cases}
C_I\xi,&0<\xi<\delta,\\
C_I\xi^{1-\alpha}\omega(\xi),&\xi\geq\delta.
\end{cases}\label{eq:I1homo}\\
II_1\leq&~\omega(\xi)A(\xi),\quad\text{where }
A(\xi) \leq
\begin{cases}
C_{II},&0<\xi<\delta,\\
C_{II}\gamma\xi^{-\alpha},&\xi\geq\delta.
\end{cases}\label{eq:I2homo}\\
III_1\leq&-\rho_mD(\xi),\quad\text{where }
D(\xi) \geq
\begin{cases}
C_{III}\xi^{1-\alpha/2},&0<\xi<\delta,\\
C_{III}\omega(\xi)\xi^{-\alpha},&\xi\geq\delta.
\end{cases}\label{eq:I3homo}
\end{align}
\end{lemma}

Applying the proof of Lemma \ref{lem:homopart1} to the modulus of
continuity $\omega_B$, we get the following estimates.
\begin{lemma}\label{lem:homopart}
Let $\rho(\cdot,t)$ obey the modulus of continuity $\omega_B$ as in
\eqref{eq:omegaB} for $0\leq t<t_1\leq T$, and let $x_1$, $y_1$
be the breakthrough points at the first breakthrough time $t_1$, as in
\eqref{eq:breakthrough}.
Suppose $\delta$ and $\gamma$ are small constants satisfying
\eqref{eq:deltagamma}.
Then there exist positive constants $C_2$ and $C_3$, which may only
depend on $\alpha$, such that
\begin{equation}\label{eq:I1I2}
|I_1|, ~II_1\leq\begin{cases}C_2B^{1+\alpha}\xi,&0<\xi<B^{-1}\delta,\\
C_2\gamma\omega_B(\xi)\xi^{-\alpha},&\xi\geq B^{-1}\delta,
\end{cases}
\end{equation}
and
\begin{equation}\label{eq:I3}
III_1\leq -\rho_m D_B(\xi),~~
D_B(\xi):=\begin{cases}C_3B^{1+{\alpha}/{2}}\xi^{1-{\alpha}/{2}},&0<\xi< B^{-1}\delta,\\
C_3{\omega_B(\xi)}{\xi^{-\alpha}},&\xi\geq B^{-1}\delta.
\end{cases}
\end{equation}
\end{lemma}
\begin{proof}
Through the same proof of Lemma \ref{lem:homopart1} and replacing
$\omega$ by $\omega_B$, one can obtain the following estimates similar to
\eqref{eq:I1homo}, \eqref{eq:I2homo} and \eqref{eq:I3homo}:
\[|I_1|\leq\omega_B'(\xi)\Omega_B(\xi),\quad
II_1\leq\omega_B(\xi)A_B(\xi),\quad
III_1\leq-\rho_mD_B(\xi).\]
Here $\omega_B$ is defined in \eqref{eq:omegaB}, and
\[\Omega_B(\xi)=B^{\alpha-1}\Omega(B\xi),\quad
A_B(\xi)=B^\alpha A(B\xi),\quad
D_B(\xi)=B^\alpha D(B\xi).\]
This directly implies \eqref{eq:I1I2} and \eqref{eq:I3} with
$C_2=\max\{C_I, C_{II}\}$ and $C_3=C_{III}$.
\end{proof}

If we pick $\delta$ small enough so that
\begin{equation}\label{eq:delta}
\delta<\left(\frac{C_3}{4C_2}\rho_m(T)\right)^{2/\alpha},
\end{equation}
then
\[
C_2B^{1+\alpha}\xi\leq C_2
B^{1+\alpha/2}\xi^{1-\alpha/2}\delta^{\alpha/2}\leq\frac{1}{4}\rho_mD_B(\xi),\quad
\forall~\xi\in(0,B^{-1}\delta).
\]
Also, pick $\gamma$ small enough so that
\begin{equation}\label{eq:gamma}
\gamma<\frac{C_3}{4C_2}\rho_m(T),
\end{equation}
then
\[C_2\gamma\omega_B(\xi)\xi^{-\alpha}\leq\frac{1}{4}\rho_mD_B(\xi),\quad
\forall~\xi\geq B^{-1}\delta.\]
Therefore, we have
\begin{equation}\label{eq:est1}
I_1+II_1+III_1\leq -\frac{1}{2}\rho_mD_B(\xi).
\end{equation}

It remains to control $I_2, II_2$ and $III_3$. We start with the
estimate on $I_2$.

\begin{lemma}\label{lem:drift}
Let $\rho(\cdot,t)$ obey the modulus of continuity $\omega_B$ as in
\eqref{eq:omegaB} for $0\leq t<t_1\leq T$, and let $x_1$, $y_1$
be the breakthrough points at the first breakthrough time $t_1$, as in
\eqref{eq:breakthrough}.
Suppose $\delta$ and $\gamma$ satisfy \eqref{eq:deltagamma}, and in addition
\begin{equation}\label{eq:deltagamma2}
\delta<
\left(\frac{\rho_m(T)C_3}{6\rho_M(T)F_M(T)}\right)^{2/\alpha},
\gamma<\alpha(\delta-\delta^{1+\alpha/2}),
\text{and }
B>\max\left\{1, 2\delta \exp\left(\frac{6\rho_M(T)F_M(T)}{C_3\rho_m(T)}\right)\right\}.
\end{equation}
Then,
\begin{equation}\label{eq:II1}
|I_2|\leq\frac{1}{6}\rho_mD_B(\xi).
\end{equation}
\end{lemma}
\begin{proof}
We start with the estimate
\[|I_2|\leq\|\pa_xu_2\|_{L^\infty}\xi\omega_B'(\xi)=\|G\|_{L^\infty}\xi\omega_B'(\xi)
\leq \rho_M(T)F_M(T) \xi\omega_B'(\xi).\]

For $\xi\in(0,B^{-1}\delta)$, $\omega_B'(\xi)<B$. So,
\begin{equation}\label{eq:II1a}
|I_2|\leq \rho_M(T)F_M(T)B\xi\leq\frac{1}{6}\rho_mD_B(\xi),
\end{equation}
provided that $\delta$ is small enough, satisfying
\eqref{eq:deltagamma2}, and $B>1$.

For $\xi\geq B^{-1}\delta$, since $\rho$ is periodic and $\omega_B$ is
increasing, the breakthrough can not happen first at $\xi>1/2$. So we
only need to consider $\xi\in(B^{-1}\delta, 1/2]$.
As $\omega_B'(\xi)=\frac{\gamma}{\xi}$ in this range, we get
\begin{equation}\label{eq:II1b}
|I_2|\leq \rho_M(T)F_M(T)\gamma.
\end{equation}
On the other hand, compute
\[\frac{d}{d\xi}D_B(\xi)=C_3\xi^{-\alpha-1}(-\alpha\omega_B(\xi)+\gamma)\leq C_3\xi^{-\alpha-1}(-\alpha(\delta-\delta^{1+\alpha/2})+\gamma)<0,
\]
for all $\xi\geq B^{-1}\delta$, provided that $\gamma$ is small
enough, satisfying \eqref{eq:deltagamma2}.
Therefore,
\begin{equation}\label{eq:minDBxi}
\min_{B^{-1}\delta\leq\xi\leq
    1/2}D_B(\xi)=D_B(1/2)\geq C_3\gamma\log\left(\frac{B}{2\delta}\right).
\end{equation}
Combining \eqref{eq:II1b}, \eqref{eq:minDBxi} and the assumption
on $B$ in \eqref{eq:deltagamma2}, we conclude
\begin{equation}\label{eq:II1b2}
|I_2|\leq
  \rho_M(T)F_M(T)\gamma\leq\frac{C_3}{6}\gamma\rho_m(T)\log\left(\frac{B}{2\delta}\right)
\leq\frac{1}{6}\rho_mD_B(\xi).
\end{equation}
\end{proof}

The estimates on $II_2$ and $III_2$ are more subtle.
To proceed, it is convinient to decompose $II_2+III_2$ in an
alternative way
\begin{equation}\label{eq:GtoF}
\begin{split}
II_2+III_2=&-\big(\rho(x_1)\pa_xu_2(x_1)-\rho(y_1)\pa_xu_2(y_1)\big)
=-\big(\rho(x_1)^2F(x_1)-\rho(y_1)^2F(y_1)\big)\\
=&-\big(\rho(x_1)^2-\rho(y_1)^2\big)F(x_1)-\rho(y_1)^2\big(F(x_1)-F(y_1)\big)
=IV+V.
\end{split}
\end{equation}

We first consider the case when $\xi< B^{-1}\delta$.
For $IV$, the estimate is similar to \eqref{eq:II1a}
\begin{equation}\label{eq:estIV}
|IV|=\omega_B(\xi)(\rho(x_1)+\rho(y_1))|F(x_1)|\leq
  2\rho_MF_MB\xi\leq\frac{1}{6}\rho_mD_B(\xi),
\end{equation}
where the last inequality holds if
$\delta$ is picked to be small enough, satisfying
\begin{equation}\label{eq:delta3}
\delta<\left(\frac{C_3\rho_m(T)}{12\rho_M(T)F_M(T)}\right)^{2/\alpha}.
\end{equation}

For $V$, we need the following lemma.
\begin{lemma}\label{lem:FxFy}
Let $\rho(\cdot,t)$ obey the modulus of continuity $\omega_B$  with
any $B>1$ as in
\eqref{eq:omegaB} for $0\leq t<t_1\leq T$.
Then, there exists a constant $C_F=C_F(T)$ such that
\begin{equation}\label{eq:FxFy}
|F(x,t)-F(y,t)|\leq C_F(T)B|x-y|,\quad
\forall~x,y\in\T,\quad \forall~ t\in[0,t_1].
\end{equation}
\end{lemma}
\begin{proof}
Recall the dynamics of $F$
\begin{equation}\label{eq:F2}
\pa_tF+u\pa_xF=-k\left(1-\frac{\bar{\rho}}{\rho}\right).
\end{equation}
Let $f=\pa_xF$. Differentiate \eqref{eq:F2} with respect to $x$ and
get
\begin{equation}\label{eq:Fx}
\partial_tf+\pa_x(uf)=-k\bar{\rho}\frac{\pa_x\rho}{\rho^2}.
\end{equation}
Let $q=f/\rho$. Using \eqref{eq:EPASrho} and \eqref{eq:Fx}, we obtain
\begin{equation}\label{eq:q}
\partial_tq+u\pa_xq=-k\bar{\rho}\frac{\pa_x\rho}{\rho^3}.
\end{equation}
It follows that
\[q(X(x,t),t)=q_0(x)-k\bar{\rho}\int_0^t
\frac{\pa_x\rho(X(x,s),s)}{\rho(X(x,s),s)^3}ds,\]
where $X$ is the trajectory of the characteristic path defined in \eqref{eq:path}.
Then, since for $t\leq t_1$, $\rho(\cdot,t)$ obeys $\omega_B$, we
obtain the following estimate
\begin{equation}\label{eq:qmax}
\|q(\cdot,t)\|_{L^\infty}\leq\|q_0\|_{L^\infty}+|k|\bar{\rho}\int_0^t\frac{B}{\rho_m(s)^3}ds
\leq C'(T)B,
\end{equation}
where the finite constant $C'$ depends on $T$ and initial
data. This implies
\[
|F(x)-F(y)|\leq\|f\|_{L^\infty}|x-y|\leq\rho_M(T) C'(T)B\xi=:C_F(T)B|x-y|.
\]
\end{proof}

Applying the estimate \eqref{eq:FxFy} at the breakthrough points and using
the upper bound on $\rho$ \eqref{eq:upperbound}, we get
\begin{equation}\label{eq:estV}
|V|\leq \rho_M(T)^2C_F(T)B\xi< \frac{1}{6}\rho_mD_B(\xi),
\end{equation}
where the second inequality holds by picking sufficiently small
$\delta$, satisifying
\begin{equation}\label{eq:delta4}
\delta<\left(\frac{C_3\rho_m(T)}{6\rho_M(T)^2C_F(T)}\right)^{2/\alpha},
\end{equation}
similar to the estimate in \eqref{eq:II1a}.

Combining \eqref{eq:est1}, \eqref{eq:II1a}, \eqref{eq:estIV} and
\eqref{eq:estV}, we conclude that
\[\pa_t(\rho(x_1)-\rho(x_2))<0, \quad
\forall~\xi=|x_1-x_2|<B^{-1}\delta.\]

Finally, we estimate $II_2+III_2$ for $\xi\in[B^{-1}\delta,1/2]$.
As $\rho$ and $F$ are bounded, it is clear that
\begin{equation}\label{eq:estIVV}
|II_2+III_2|\leq 2\rho_M(T)^2 F_M(T)<\frac{1}{3}\rho_mD_B(\xi).
\end{equation}
The second inequality holds by picking $B$ large enough. This is due
to the fact that $D_B(\xi)$ is an increasing in $B$ with
$\lim_{B\to\infty}D_B(\xi)=\infty$. More precisely,
using the bound \eqref{eq:minDBxi}, it suffices to pick
\begin{equation}\label{eq:B2}
B>2\delta\exp\left(\frac{6\rho_M(T)^2 F_M(T)}{C_3\gamma\rho_m(T)}\right).
\end{equation}

Combining \eqref{eq:est1}, \eqref{eq:II1b2} and
\eqref{eq:estIVV}, we conclude that
\[\pa_t(\rho(x_1)-\rho(x_2))<0, \quad
\forall~\xi=|x_1-x_2|\in[B^{-1}\delta,1/2].\]

Let us summerize the procedure on the construction of the modulus of
continuity $\omega_B$. First, we fix a time $T$.
Then, we pick a small parameter $\delta$
satisfying \eqref{eq:deltagammaBinit}, \eqref{eq:deltagamma},
\eqref{eq:delta}, \eqref{eq:deltagamma2}, \eqref{eq:delta3} and \eqref{eq:delta4}:
\begin{equation}\label{eq:deltaall}
\delta<\min\left\{1,
\frac{2\|\rho_0\|_{L^\infty}}{\|\pa_x\rho_0\|_{L^\infty}},
\left(\frac{C_3\rho_m(T)}{\max\{4C_2,
    12\rho_M(T)F_M(T),  6\rho_M(T)^2C_F(T)\}}\right)^{2/\alpha}\right\}.
\end{equation}
Next, we pick a small parameter $\gamma$ satisfying
\eqref{eq:deltagamma}, \eqref{eq:gamma} and \eqref{eq:deltagamma2}:
\begin{equation}\label{eq:gammaall}
\gamma<\min\left\{
\frac{C_3}{4C_2}\rho_m(T),~\min\left(\frac{1}{2\log2},\alpha\right)\cdot
(\delta-\delta^{1+\alpha/2})
\right\}.
\end{equation}
Finally, we pick a large parameter $B$ satisfying
\eqref{eq:deltagammaBinit}, \eqref{eq:deltagamma2} and \eqref{eq:B2}:
\begin{equation}\label{eq:Ball}
B>\max\left\{1,
  \frac{\delta\|\pa_x\rho_0\|_{L^\infty}}{2\|\rho_0\|_{L^\infty}}\exp\left(\frac{2\|\rho_0\|_{L^\infty}}{\gamma}\right),
  2\delta\exp\left(\frac{6\rho_M(T)^2 F_M(T)}{C_3\gamma\rho_m(T)}\right)\right\}.
\end{equation}
Here we assume, without loss of generality, that $\gamma \leq 1$ and $\rho_M(T) \geq 1$ to simplify
the expression.

With this choice of $\omega_B$, we have shown that $\rho(\cdot, t)$
obeys $\omega_B$  for all $t\in[0,T]$. Hence,
\[\|\pa_x\rho(\cdot, t)\|_{L^\infty}\leq B,\quad\forall~t\in[0,T].\]
Therefore, condition \eqref{eq:BKM} is satisfied, and we obtain global
regularity of the system.

We end this section by the following remark.
\begin{remark}
When $k=0$, all the quantities $\rho_m, \rho_M, F_M$ and $C_F$ do not
depend on $T$. As a consequence, $\delta, \gamma$
and $B$ do not depend on $T$ as well. Therefore,
$\|\pa_x\rho(\cdot, t)\|_{L^\infty}\leq B$ for any $t\geq0$. This
estimate improves the result obtained in \cite{do2017global},
where the bound on $\pa_x\rho$ could grow in time. We note that
stationary in time bound on $\pa_x\rho$ for the Euler-Alignment model
(without Possion forcing) has been derived in \cite{shvydkoy2017eulerian3} by a different
argument.

For $k\neq 0$, with the singular attractive or repulsive force, our
estimate on $\rho_m$ and $\rho_M$ is not uniform in time. We are able
to obtain time-dependent bounds \eqref{eq:lowerbound} and
\eqref{eq:upperbound}, where $\rho_m$ can decay exponentially in time,
and $\rho_M$ (and $F_M, C_F$) can grow exponentially in time. From
\eqref{eq:deltaall} and \eqref{eq:gammaall}, we see that $\delta$ and
$\gamma$ decay exponentially in time. Finally, from
\eqref{eq:Ball}, $B$ grows double exponentially in time. Therefore,
we obtain a double exponential in time bound on
$\|\pa_x\rho(\cdot,t)\|_{L^\infty}$. It is not clear whether such
bound is optimal. We will leave it for future investigation.
\end{remark}

\section{Euler dynamics with general three-zone interactions}
In this section, we extend our global regularity result for EPA system
to more general Euler dynamics with three-zone interactions. Recall
the \emph{Euler-3Zone system} under periodic setup
\begin{align}
&\pa_t \rho + \pa_x (\rho u) = 0, \quad x\in\T,\,\,t>0,\label{eq:grho} \\
&\,\,\,\, \pa_t u + u\pa_x u =  \int_{\T} \psi(y)(u(x+y,t) - u(x,t))\rho(y,t)dy
- \pa_x K \star \rho.\label{eq:gu}
\end{align}
We will discuss the global wellposedness of the system with more general
singular influence function $\psi$ and interaction potential $K$.

\subsection{General singular influence function}
Consider a general influence function $\psi$ which is positive
\begin{equation}\label{eq:psimg}
\psi_m:=\min_{x\in\T}\psi(x)>0,
\end{equation}
and singular at origin.
Recall the decomposition \eqref{eq:psidecomp}:
we will consider the class of functions where one
can decompose $\psi$ into two parts
\begin{equation}\label{eq:psidecomp2}
\psi=c\psi_\alpha+\psi_L.
\end{equation}
Here $\psi_\alpha$ is the singular power defined in \eqref{eq:psip}, and $\psi_L$ is
bounded and Lipschitz.
In this case, let
\begin{equation}\label{eq:gg}
G=\pa_xu-c\Lambda^\alpha\rho+\psi_L\star\rho.
\end{equation}
Then, the dynamics of $G$ reads:
\begin{align*}
\pa_tG=&~\pa_t\pa_xu-c\pa_t\Lambda^\alpha\rho+\psi\star\pa_t\rho\\
=&-\pa_x(u\pa_xu)+c\pa_x\big(-\Lambda^\alpha(\rho
   u)+u\Lambda^\alpha\rho\big)-\pa_x\big(\psi_L\star(\rho
   u)-u(\psi_L\star\rho)\big)-\pa_{xx}^2K\star\rho\\
&+c\Lambda^\alpha\pa_x(\rho u)-\psi_L\star\pa_x(\rho u)\\
=&-\pa_x\big(u(\pa_xu-c\Lambda^\alpha\rho+\psi_L\star\rho)\big)-\pa_{xx}^2K\star\rho=-\pa_x(Gu)-\pa_{xx}^2K\star\rho.
\end{align*}
Therefore, $(\rho,G)$ still satisfy \eqref{eq:EPASrho2} and
\eqref{eq:EPASG},
\begin{equation}\label{eq:rhoandg}
\pa_t\rho+\pa_x(\rho u) = 0,\quad
\pa_tG+\pa_x(G u) = -\pa_{xx}^2K\star\rho,
\end{equation}
with a different relation
\begin{equation}\label{eq:uxg}
\pa_xu=\Lambda^\alpha\rho+G-\psi_L\star\rho.
\end{equation}
Then the velocity field $u$ can be recovered as
\begin{equation}\label{eq:urecg}
u(x,t)=\Lambda^\alpha\partial_x^{-1}(\rho(x,t)-\bar{\rho})+
\partial_x^{-1}\big(G(x,t)-\psi_L\star\rho(x,t)\big)+I_0(t),
\end{equation}
where $I_0(t)$ can be determined by conservation of momentum
\eqref{eq:mom}. The second term on the right hand side is well-defined since
\[\int_\T \big(G(x,t)-\psi_L\star\rho(x,t)\big)dx=
\int_\T\big(\pa_xu(x,t)-\Lambda^\alpha\rho(x,t)\big)dx=
0,\quad\forall~t\geq0.\]
We can decompose $u$ into two parts  $u=u_S+u_L$, where $u_S$ is the
singular part
\begin{equation}\label{eq:using}
\begin{split}
u_S(x,t)=&\Lambda^\alpha\partial_x^{-1}(\rho(x,t)-\bar{\rho})+
\partial_x^{-1}\left(G(x,t)-\int_\T G(x,0)dx\right),\\
\pa_xu_S=&\Lambda^\alpha\rho+G-\int_\T G(x,0)dx,
\end{split}
\end{equation}
and $u_L$ is the Lipschitz part
\begin{equation}\label{eq:uLip}
\begin{split}
u_L(x,t)=&-
\partial_x^{-1}\left(\psi_L\star\rho(x,t)-\int_\T
  G(x,0)dx\right)+I_0(t),\\
\pa_xu_L=&-\psi_L\star\rho+\int_\T G(x,0)dx.
\end{split}
\end{equation}

Now, we follow the same procedure as fractional EPA system to show
global regularity of system \eqref{eq:rhoandg}, \eqref{eq:urecg}. We
first take the Newtonian potential \eqref{eq:Newtonian}. General interaction
potentials will be discussed in the next section. The arguments below follow
the same outline, so we focus on indicating changes.

\subsubsection*{Step 1: Apriori lower bound on $\rho$}
The statement and proof are identical to Theorem \ref{thm:lower},
except that estimate \eqref{eq:lb1} is replaced by
\begin{equation}\label{eq:lbg}
\begin{aligned}
-c\Lambda^\alpha\rho(\underline{x},t)+&\psi_L\star\rho(\underline{x},t)\\
=&~\int_\T\big(c\psi_\alpha(y)+\psi_L(y)\big)\big(\rho(\underline{x}-y,t)-\rho(\underline{x},t)\big)dy+\rho(\underline{x},t)\int_\T\psi_L(y)dy\\
\geq&~
\psi_m\int_\T\big(\rho(\underline{x}-y,t)-\rho(\underline{x},t)\big)dy
-\rho(\underline{x},t)\|\psi_L\|_{L^\infty}\\
=&~\psi_m\bar{\rho}-(\psi_m+\|\psi_L\|_{L^\infty})\rho(\underline{x},t).
\end{aligned}
\end{equation}
Hence, estimate \eqref{eq:rhombound1} becomes
\begin{equation}\label{eq:rhomboundg}
\pa_t\rho(\underline{x},t)\geq~
\big(\psi_m\bar{\rho}\big)\rho(\underline{x},t)-\left[\psi_m+\|\psi_L\|_{L^\infty}+
\|F_0\|_{L^\infty}+|k|t+|k|\bar{\rho}\int_0^t\frac{1}{\rho_m(s)}ds\right]\rho(\underline{x},t)^2,
\end{equation}
where the only extra term
$\|\psi_L\|_{L^\infty}\rho(\underline{x},t)^2$ is quadratic in
$\rho$, and can be
controlled by the linear term
$\big(\psi_m\bar{\rho}\big)\rho(\underline{x},t)$ if $\rho_m$ is small
enough.

Following the same proof, we obtain the lower bound \eqref{eq:lowerbound}
with coefficient $A_m, C_m$ satisfying \eqref{eq:AC} for any
$\epsilon\in(0,\epsilon_*)$, where $\epsilon_*=\frac{1}{2}\left(\sqrt{1+\frac{4\psi_m}{e(\psi_m+\|\psi_L\|_{L^\infty}+\|F_0\|_{L^\infty})}}-1\right)$.

\subsubsection*{Step 2: Apriori upper bound on $\rho$}
We follow the proof of Theorem \ref{thm:upper}. The estimate
\eqref{eq:upperbound1} becomes
\begin{align*}
\frac{d}{dt}\rho(\bar{x},t)\leq& -C_1\rho(\bar{x},t)^{2+\alpha}+
F_M(t)\rho(\bar{x},t)^2+\rho(\bar{x},t)\cdot\psi_L\star\rho(\bar{x},t)\\
\leq& -C_1\rho(\bar{x},t)^{2+\alpha}+
F_M(t)\rho(\bar{x},t)^2+\|\psi_L\|_{L^\infty}\bar{\rho}\rho(\bar{x},t).
\end{align*}
Both second and third terms are dominated by the first term if
$\rho(\bar{x},t)$ is big enough. In particular $\pa_t\rho(\bar{x},t)<0$ if
$\rho(\bar{x},t)>\max\{(2F_M/C_1)^{1/\alpha},
(2\|\psi_L\|_{L^\infty}\bar{\rho}/C_1)^{1/(1+\alpha)}\}$
. Therefore,
\begin{equation}\label{eq:rhomaxbg}
\rho(x,t)\leq\rho(\bar{x},t)\leq\max\left\{\|\rho_0\|_{L^\infty},
  3\bar{\rho}, \left(\frac{2F_M(t)}{C_1}\right)^{1/\alpha},
\left(\frac{2\|\psi_L\|_{L^\infty}\bar{\rho}}{C_1}\right)^{1/(1+\alpha)}\right\},
\end{equation}
and \eqref{eq:upperbound} holds with $A_M=A_m/\alpha$ and
\[C_M=\max\left\{\max_{x\in\T}\rho_0(x), ~ 3\bar{\rho}, ~
\left[\frac{2}{C_1}\left(\|F_0\|_{L^\infty}+\frac{|k|}{eA_m}+\frac{|k|\bar{\rho}}{A_mC_m}\right)\right]^{\frac{1}{\alpha}},
\left(\frac{2\|\psi_L\|_{L^\infty}\bar{\rho}}{C_1}\right)^{\frac{1}{1+\alpha}}\right\}.\]

\subsubsection*{Step 3: Local wellposedness}
We write the system \eqref{eq:rhoandg} \eqref{eq:urecg} in terms of
$\theta=\rho-\bar{\rho}$ and $G$ as follows
\begin{align}
\pa_t\theta+\pa_x(\theta u_S)+\pa_x(\theta u_L)&=-\bar{\rho}\pa_xu_S-\bar{\rho}\pa_xu_L,\label{eq:thetag}\\
\pa_tG+\pa_x(Gu_S)+ \pa_x(G u_L)&=-k\theta,\label{eq:localgg}
\end{align}
where $u_S$ and $u_L$ are defined in \eqref{eq:using} and
\eqref{eq:uLip} respectively.

We proceed with a Gronwall estimate on the quantity $Y$ in
\eqref{eq:Y}.
The estimates in Theorem \ref{thm:local} can be applied directly to
the $u_S$ part. We will focus on the Lipschitz part $u_L$. The
procedure is similar to \cite[Theorem Appendix A.1]{carrillo2016critical}. We
will summarize in below.

For the term $\pa_x(\theta u_L)$, we have
\[
\int_\T\Lambda^s\theta\cdot\Lambda^s\pa_x(\theta u_L)dx
=\int_\T\Lambda^s\theta\cdot\Lambda^s\pa_x\theta\cdot u_Ldx
+\int_\T\Lambda^s\theta\cdot[\Lambda^s\pa_x, u_L]\theta~dx
=:L_1+L_2.\]
We estimate the two terms one by one. For $L_1$,
\begin{equation}\label{eq:l1est}
|L_1|=\left|\int_\T\pa_x\left(\frac{(\Lambda^s\theta)^2}{2}\right)
u_Ldx\right|\leq
\frac{1}{2}\int_\T(\Lambda^s\theta)^2|\pa_xu_L| dx
\leq\frac{1}{2}\|\psi_L\|_{L^\infty}\bar{\rho}\|\theta\|_{H^s}^2.
\end{equation}
For $L_2$, we apply commutator estimate (e.g. \cite[Lemma Appendix
A.1]{carrillo2016critical}) and get
\begin{align}
|L_2|\leq&~\|\theta\|_{H^s}\big\|[\Lambda^s\pa_x,
u_L]\theta\big\|_{L^2}
\lesssim\|\theta\|_{H^s}\left(\|\pa_xu_L\|_{L^\infty}\|\theta\|_{H^s}+
\|\pa_xu_L\|_{H^{s}}\|\theta\|_{L^\infty}\right)\nonumber\\
\leq&~\|\theta\|_{H^s}\left(\|\psi_L\|_{L^\infty}\bar{\rho}\|\theta\|_{H^s}+
\|\psi_L\|_{L^\infty}\|\rho\|_{H^{s}}\|\theta\|_{L^\infty}\right)\nonumber\\
\leq&~\|\psi_L\|_{L^\infty}(2\bar{\rho}+\|\theta\|_{L^\infty})\|\theta\|_{H^s}^2
+\frac{1}{4}\|\psi_L\|_{L^\infty}\bar{\rho}.\label{eq:l2est}
\end{align}
Note that for the last inequality, we have used
$\|\rho\|_{H^s}\leq \|\theta\|_{H^s}+\|\bar{\rho}\|_{H^s}=
 \|\theta\|_{H^s}+\bar{\rho}$.

For the term $-\bar{\rho}\pa_xu_L$,
\begin{equation}\label{eq:l3est}
\left|-\bar{\rho}\int_\T\Lambda^s\theta\cdot\Lambda^s\pa_x u_Ldx\right|
\leq
\bar{\rho}\|\theta\|_{H^s}\|(\pa_x\psi_L)\star(\Lambda^s\rho)\|_{L^2}
\leq
\bar{\rho}\|\pa_x\psi_L\|_{L^\infty}\|\theta\|_{H^s}(\|\theta\|_{H^s}
+\bar{\rho}).
\end{equation}

For the term $\pa_x(Gu_L)$, the estimate is similar to the term
$\pa_x(\theta u_L)$.
\begin{align*}
\int_\T\Lambda^{s-\frac{\alpha}{2}}G&\cdot\Lambda^{s-\frac{\alpha}{2}}
\pa_x (Gu_L)dx\\
=&\int_\T\Lambda^{s-\frac{\alpha}{2}}G\cdot\Lambda^{s-\frac{\alpha}{2}}
\pa_x G\cdot u_Ldx
+\int_\T\Lambda^{s-\frac{\alpha}{2}}G\cdot
[\Lambda^{s-\frac{\alpha}{2}}\pa_x,u_L]G~dx
=:L_4+L_5.
\end{align*}
where
\begin{equation}\label{eq:l4est}
|L_4|=\left|\int_\T\pa_x\left(\frac{(\Lambda^{s-\frac{\alpha}{2}}G)^2}{2}\right)
u_Ldx\right|\leq
\frac{1}{2}\int_\T(\Lambda^{s-\frac{\alpha}{2}}G)^2|\pa_xu_L| dx
\leq\frac{1}{2}\|\psi_L\|_{L^\infty}\bar{\rho}\|G\|_{H^{s-\frac{\alpha}{2}}}^2,
\end{equation}
and
\begin{align}
|L_5|\leq&~\|G\|_{H^{s-\frac{\alpha}{2}}}\big\|[\Lambda^{s-\frac{\alpha}{2}}\pa_x,
u_L]G\big\|_{L^2}
\lesssim\|G\|_{H^{s-\frac{\alpha}{2}}}\left(\|\pa_xu_L\|_{L^\infty}
\|G\|_{H^{s-\frac{\alpha}{2}}}+
\|\pa_xu_L\|_{H^{s}}\|G\|_{L^\infty}\right)\nonumber\\
\leq&~\|\psi_L\|_{L^\infty}\bar{\rho}\|G\|_{H^{s-\frac{\alpha}{2}}}^2
+\|\psi_L\|_{L^\infty}\bar{\rho}(\|\theta\|_{H^s}+\bar{\rho})
\|G\|_{H^{s-\frac{\alpha}{2}}}\nonumber\\
\leq&~\|\psi_L\|_{L^\infty}\bar{\rho}\left[2\|G\|_{H^{s-\frac{\alpha}{2}}}^2
+\frac12\|\theta\|_{H^{s-\frac{\alpha}{2}}}^2+\frac{\bar{\rho}^2}{2}
\right].\label{eq:l5est}
\end{align}

Combining \eqref{eq:YHs}, \eqref{eq:l1est}, \eqref{eq:l2est}, \eqref{eq:l3est},
\eqref{eq:l4est}, \eqref{eq:l5est} and the fact that
$\|G(\cdot,t)\|_{L^\infty}$ is controlled from above by a finite (growing in time) bound, we obtain that for all $t\in[0,T]$,
\begin{equation}\label{eq:YHsg}
\frac{d}{dt}Y(t)\leq C(T)(1+\|\pa_x\theta(\cdot,t)\|_{L^\infty}^2)Y(t)-\frac{\rho_m(t)}{6}\|\theta\|_{H^{s+\frac{\alpha}{2}}}^2,
\end{equation}
where the constant $C$ depends on initial data and $T$.
The same Gronwall's inequality yields local wellposedness as well as
BKM-type blowup condition \eqref{eq:BKM}.

\subsubsection*{Step 4: Global wellposedness}
To check the condition \eqref{eq:BKM}, we will use the procedure identical
to that in section \ref{sec:global}. Let us decompose $u$ as in
\eqref{eq:urecg}, $u = u_1+u_2$ where
\begin{equation}\label{eq:u1u2g}
u_1(x,t)=\Lambda^\alpha\partial_x^{-1}(\rho(x,t)-\bar{\rho}),\quad
u_2(x,t)=\partial_x^{-1}\big(G(x,t)-\psi_L\star\rho(x,t)\big)+I_0(t).
\end{equation}
The only difference between our system \eqref{eq:rhoandg}
\eqref{eq:urecg} and the EPA system is that there is an extra term in
$u_2$.
Throughout the proof in section \ref{sec:global}, the only property of
$u_2$ we have used is that $\pa_xu_2$ is bounded, namely
\[\|\pa_xu_2(\cdot,t)\|_{L^\infty}\leq \rho_M(T)F_M(T)<\infty,\quad\forall~t\in[0,T].\]
For our $u_2$ defined in \eqref{eq:u1u2g}, we also have a bound on
$\pa_xu_2$:
\[\|\pa_xu_2(\cdot,t)\|_{L^\infty}=\|G(\cdot,t)-\psi_L\star\rho(\cdot,t)\|_{L^\infty}\leq
  \rho_M(T)F_M(T)+\|\psi_L\|_{L^\infty}\bar\rho<\infty,\quad\forall~t\in[0,T].\]
Hence, global regularity follows from the same procedure by
controlling the modulus of continuity.

\subsection{General interaction potential}
In this part, we consider system \eqref{eq:grho}-\eqref{eq:gu} with a
general interaction potential $K\in W^{2,\infty}(\T)$. This class of
potentials is more regular than the Newtonian potential $\mathcal{N}$
defined in \eqref{eq:Newtonian}, as
$\pa^2_{xx}\mathcal{N}=k(\delta_0-1)\not\in L^\infty$, where
$\delta_0$ is the Dirac delta at $x=0$. We will show global
wellposedness of Euler-3Zone system with $W^{2,\infty}$
potentials. The result automatically extends to systems with
potentials that can be decomposed into a sum of a Newtonian potential
and a $W^{2,\infty}$ potential.

Now, let us assume $K\in W^{2,\infty}(\T)$. After the transformation,
the dynamics for $(\rho,G)$ becomes \eqref{eq:rhoandg}, with velocity
field $u$ defined as \eqref{eq:urecg}. We shall run through the same
procedure and point out the differences.

\subsubsection*{Step 1: Apriori lower bound on $\rho$}
Due to the change of the potential, the dynamics of $F$
\eqref{eq:Fcha} becomes
\begin{equation}\label{eq:Fchag}
(\pa_t+u\pa_x)F=-\frac{\pa_{xx}^2K\star\rho}{\rho}.
\end{equation}
Therefore, we get
\begin{equation}\label{eq:Fg}
F(X(x,t),t)=F_0(x)-\int_0^t\frac{\pa_{xx}^2K\star\rho(X(x,s),s)}{\rho(X(x,s),s)}ds,
\end{equation}
where $X(x,t)$ is the characteristic path defined in \eqref{eq:path}.
Then, we obtain a bound
\begin{equation}\label{eq:Fboundg}
\|F(\cdot,t)\|_{L^\infty}\leq~\|F_0\|_{L^\infty}+\|\pa_{xx}K\|_{L^\infty}\bar{\rho}
\int_0^t\frac{1}{\rho_m(s)}ds,
\end{equation}
which is similar as \eqref{eq:Fbound}. In fact, it is a simpler
bound as the right hand side does not contain a linear term on $t$.

The lower bound \eqref{eq:lowerbound} follows then by the same
argument, with
\[A_m=\frac{\|\pa_{xx}^2K\|_{L^\infty}}{\psi_m}, \quad
C_m=\min\left\{\rho_m(0),\frac{\psi_m\bar{\rho}}{\psi_m+\|\psi_L\|_{L^\infty}+\|F_0\|_{L^\infty}}\right\}.\]

\subsubsection*{Step 2: Apriori upper bound on $\rho$}
The upper bound estimate \eqref{eq:rhomaxbg} can be obtained without any
additional difficulties. Since we have
\[F_M(t)=\|F_0\|_{L^\infty}+\frac{\|\pa_{xx}^2K\|_{L^\infty}\bar{\rho}}{A_mC_m}e^{A_mt}\]
by \eqref{eq:Fboundg} and the lower bound estimate on $\rho$,
the upper bound \eqref{eq:upperbound} holds with
$A_M=A_m/\alpha$ and
\[C_M=\max\left\{\max_{x\in\T}\rho_0(x), ~ 3\bar{\rho}, ~
\left[\frac{2}{C_1}\left(\|F_0\|_{L^\infty}+\frac{\|\pa_{xx}^2K\|_{L^\infty}\bar{\rho}}{A_mC_m}\right)\right]^{\frac{1}{\alpha}},
\left(\frac{2\|\psi_L\|_{L^\infty}\bar{\rho}}{C_1}\right)^{\frac{1}{1+\alpha}}\right\}.\]

\subsubsection*{Step 3: Local wellposedness}
Since the potential only enters the dynamics of $G$ equation, so the system in
terms of $(\theta, G)$ is identical to
\eqref{eq:thetag}-\eqref{eq:localgg}, except the right hand side of
\eqref{eq:localgg}
is replaced by $-\pa_{xx}^2K\star\rho$. Hence, we only need to
estimate this extra term.
\begin{align*}
\left|\int_\T\Lambda^{s-\frac{\alpha}{2}}G\right.&\left.\cdot~\Lambda^{s-\frac{\alpha}{2}}(\pa_{xx}^2K\star\rho)~dx\right|
=\left|\int_\T\Lambda^{s-\frac{\alpha}{2}}G\cdot(\pa_{xx}^2K\star\Lambda^{s-\frac{\alpha}{2}}\rho)~dx\right|\\
\lesssim&~\|\pa_{xx}^2K\|_{L^\infty}\|G\|_{H^{s-\frac{\alpha}{2}}}\|\theta\|_{H^s}
\leq \frac{1}{2}\|\pa_{xx}^2K\|_{L^\infty}Y(t).
\end{align*}
The local wellposedness and BKM-type blowup condition \eqref{eq:BKM}
follow by applying the same Gronwall's inequality on $Y$.

\subsubsection*{Step 4: Global wellposedness}
The argument for EPA system can be directly applied to the general system as
the $\rho$ equations in both cases are the same. The different
potential does change the estimates on $\rho_m, \rho_M, F_M, C_F$,
which are needed to construct the modulus $\omega_B$. Since $\rho_m,
\rho_M$ and $F_M$ have been treated in the previous steps, we are left
with estimating $C_F$, namely proving Lemma \ref{lem:FxFy} for the
general system.
\begin{proof}[Proof of Lemma \ref{lem:FxFy}]
Let $f=\pa_xF$. Differentiate \eqref{eq:Fchag} with respect to $x$ and
get
\begin{equation}\label{eq:Fxg}
\partial_tf+\pa_x(uf)=\frac{-(\pa_{xxx}^3K\star\rho)\rho+(\pa_{xx}^2K\star\rho)\pa_x\rho}{\rho^2}.
\end{equation}
Let $q=f/\rho$. Using \eqref{eq:EPASrho} and \eqref{eq:Fxg}, we obtain
\begin{equation}\label{eq:qg}
\partial_tq+u\pa_xq=\frac{-(\pa_{xxx}^3K\star\rho)\rho+(\pa_{xx}^2K\star\rho)\pa_x\rho}{\rho^3}.
\end{equation}
For $t\leq t_1$, $\rho(\cdot,t)$ obeys $\omega_B$. Then
$\|\pa_x\rho(\cdot,t)\|_{L^\infty}\leq\omega_B'(0)=B$. Therefore, we
can bound the right hand side of \eqref{eq:qg} as follows:
\begin{align*}
\left|\frac{-(\pa_{xxx}^3K\star\rho)\rho+(\pa_{xx}^2K\star\rho)\pa_x\rho}{\rho^3}\right|
\leq&~\frac{\|\pa_{xx}^2K\|_{L^\infty}\|\pa_x\rho\|_{L^1}}{\rho_m(t)^2}
+\frac{\|\pa_{xx}^2K\|_{L^\infty}\bar{\rho}\|\pa_x\rho\|_{L^\infty}}{\rho_m(t)^3}\\
\leq&~B\|\pa_{xx}^2K\|_{L^\infty}\left(\frac{1}{\rho_m(t)^2}+\frac{\bar{\rho}}{\rho_m(t)^3}\right).
\end{align*}
Then, we obtain the bound on $q$ for all $0\leq t\leq t_1<T$,
\begin{equation}\label{eq:qmaxg}
\|q(\cdot,t)\|_{L^\infty}\leq\|q_0\|_{L^\infty}+B\|\pa_{xx}^2K\|_{L^\infty}\int_0^t\left(\frac{1}{\rho_m(t)^2}+\frac{\bar{\rho}}{\rho_m(t)^3}\right)ds
\leq C'(T)B,
\end{equation}
where the finite constant $C'$ depends on $\|\pa_{xx}^2K\|_{L^\infty},
T$ and initial
data. This implies
\[
|F(x)-F(y)|\leq\|f\|_{L^\infty}|x-y|\leq\rho_M(T) C'(T)B\xi=:C_F(T)B|x-y|.
\]
\end{proof}

\textbf{Acknowledgment.} This work has been partially supported by the NSF grants DMS 1412023 and DMS 1712294.

\bibliographystyle{plain}
\bibliography{singular}

\end{document}